\documentclass[reqno]{amsart}
\usepackage{amssymb}
\usepackage{amsmath}
\usepackage{amsfonts}
\usepackage{graphicx}
\usepackage{subfig}
\usepackage{eurosym}
\usepackage{amssymb}
\usepackage{amsmath}
\usepackage{amsfonts}
\usepackage{graphicx}
\usepackage{subfig}

\newtheorem{theorem}{Theorem}
\theoremstyle{plain}

\newtheorem{case}{Case}

\newtheorem{corollary}{Corollary}

\newtheorem{definition}{Definition}

\newtheorem{remark}{Remark}

\numberwithin{equation}{section}

\begin{document}
\title[Batman University]{third-order relativistic kinematics and siacci shock decompositions via localized darboux frames}
\author{Fatma ALMAZ}
\address{department of mathematics, faculty of arts and sciences, batman
university, batman/ t\"{u}rk\.{ı}ye orcid: 0000-0002-1060-7813}
\email{fatma.almaz@batman.edu.tr}
\author{\c{C}\.{ı}\u{g}dem YAVUZ}
\address{department of mathematics, faculty of arts and sciences, batman
university, batman/ t\"{u}rk\.{ı}ye, orcid: 0009-0002-7030-4205}
\email{yavuzz.cigdemm20@gmail.com}
\thanks{This paper is in final form and no version of it will be submitted
for publication elsewhere.}

\begin{abstract}
Constrained particle trajectories along lower-dimensional worldsheets constitute a foundational domain in relativistic kinematics and submanifold theory. This paper establishes a comprehensive geometric and kinematic framework for third-order particle dynamics constrained to timelike surfaces in Minkowski 3-space $E_1^3$. By utilizing the localized Darboux frame, we explicitly derive the relativistically invariant parameters namely the geodesic curvature, normal curvature, and geodesic torsion and analyze the non-linear coupling between the intrinsic surface topology and the extrinsic spacetime geometry. Primary focus is dedicated to formulating the relativistic "jerk" vector, which accounts for the instantaneous time rate of change of acceleration and captures the high-order structural stresses of the motion. Furthermore, we generalize Siacci's theorem to this Lorentzian setting, providing a non-perpendicular geometric decomposition of both the acceleration and jerk fields into explicit tangential and radial components relative to a fixed coordinate origin. Our formulations explicitly unveil the hidden central-force dynamics and dimensional compactification mechanisms under planar constraints, showing that out of plane torsional shocks vanish identically when the modified geodesic torsion is suppressed. These results offer novel analytical insights for structural stability and conserved angular momentum like quantities governing constrained physical systems in non-Euclidean spacetimes.

\end{abstract}

\maketitle

\section{Introduction}

The profound relationship between theoretical physics and differential geometry offers an exceptionally powerful mathematical framework for deciphering the complex structure of spacetime manifolds and the trajectories of particles moving within them. In relativistic mechanics, timelike surfaces embedded in a lower-dimensional Lorentz-Minkowski space play a vital role in modeling the evolutionary worldsheets of physical phenomena emerging under the influence of gravitational fields and external forces. The trajectory of a massive or massless particle constrained to such surfaces possesses not only instantaneous kinematic characteristics but also deep dynamical information, which can be comprehensively decoded using the rich geometric machinery of sub-manifold theory.

In classical kinematics, the analysis of curve geometry often relies on the traditional Frenet reference frame. However, when particle motion is strictly constrained to a surface, the Frenet frame fails to incorporate the structural orientation of the surface itself. To bridge this geometric gap, the Darboux frame serves as an indispensable tool, providing a localized orthonormal vector basis that directly encapsulates the intrinsic and extrinsic geometric characteristics of the underlying surface. Consequently, fundamental invariants such as geodesic curvature ($k_g$), normal curvature ($k_n$), and geodesic torsion ($\tau_g$) can be precisely quantified \cite{4,12}, revealing how the local bending, tilting, and twisting behavior of a trajectory is intrinsically coupled with the surface topography.

Beyond the standard considerations of position, velocity, and acceleration, modern kinematics frequently demands the exploration of higher-order dynamic parameters. Among these, the "jerk" vector defined as the time derivative of the acceleration vector has emerged as a critical diagnostic variable. The jerk vector characterizes the abrupt variations in forces experienced by a moving system, acting as an invaluable guide in mechanism design, orbital maneuvering, and gravitational wave response where short, violent jolts govern the structural integrity of the system \cite{3,5,11}. While recent literature has explored the properties of jerk vectors using modified orthogonal or Frenet frames \cite{8,9}, a complete formulation of third-order kinematics on timelike surfaces using the Darboux frame remains an open and highly relevant problem.

Furthermore, analyzing the dynamic constraints and conserved quantities of such mechanical systems requires sophisticated projection techniques. This is precisely where Siacci's theorem, a product of the kinematic genius of the Italian mathematician Francesco Siacci, becomes paramount. Siacci originally introduced a groundbreaking resolution of the acceleration vector for planar trajectories, decomposing it into radial and tangential components that are inherently non-orthogonal but uniquely reveal the central-force geometry of the motion \cite{1,10}. Although this classical theorem has been generalized to Euclidean space curves using specialized frames \cite{6,14}, its adaptation to the pseudo-Riemannian metric of Minkowski spacetime is non-trivial due to the presence of hyperbolic angles and pseudo-rotational Lorentz boosts.

Motivated by these considerations, the primary objective of this study is to comprehensively analyze the third-order relativistic kinematics and dynamic properties of a particle moving along regular trajectories on timelike surfaces in Minkowski 3-space $\mathbb{E}_1^3$. By employing the Darboux frame, we systematically derive the algebraic formulations for both the jerk vector and the generalized Siacci components, considering the distinct geometric cases where the principal normal vector field is timelike. Through this approach, we demonstrate how the global observer orientation relative to a fixed coordinate origin cross-couples with local proper-time kinematics, providing an advanced analytical understanding of constrained particle dynamics in non-Euclidean spaces.

\section{Preliminaries}

The Minkowski 3-space $E_{1}^{3}$ is real vector space provided with the
standard flat metric given as following 
\begin{equation}
\ \left\langle ,\right\rangle =dx_{1}^{2}+dx_{2}^{2}-dx_{3}^{2},  \tag{2.1}
\end{equation}%
where $x=(x_{1},x_{2},x_{3})$ is a coordinate system of $E_{1}^{3}$. A
vector $W$ on $E_{1}^{3}$ is called spacelike if $\left\langle
W,W\right\rangle >0$ or $W=0$, timelike if \ $\left\langle W,W\right\rangle
<0$ and null if \ $\left\langle W,W\right\rangle =0$ and $W\neq 0$. \
Minkowski inner product of any two vectors $%
V=(V_{1},V_{2},V_{3}),W=(W_{1},W_{2},W_{3})\in E_{1}^{3}$ is expressed as $%
\left\langle W,V\right\rangle =W_{1}V_{1}+W_{2}V_{2}-W_{3}V_{3}$. The
Minkowski norm of any vector $W\in E_{1}^{3}$ is defined as $\left\Vert
W\right\Vert =\sqrt{\left\vert \left\langle W,W\right\rangle \right\vert }$
which equals zero if $W$ is lightlike vector or $W=0.$ Also, any two vectors 
$V,W\in E_{1}^{3}$ are orthogonal with notation $V\bot W$ if $\left\langle
W,V\right\rangle =0.$ The Lorentz vector product $V,W$ is given by%
\begin{equation}
V\times W=\left\vert 
\begin{array}{ccc}
e_{1} & e_{2} & -e_{3} \\ 
V_{1} & V_{2} & V_{3} \\ 
W_{1} & W_{2} & W_{3}%
\end{array}%
\right\vert ,  \tag{2.2}
\end{equation}%
\cite{2,7}. For a surface in Minkowski space to be timelike its normal
vector at every point on the surface must be spacelike. This means that all
vectors in the tangent plane of the surface can be timelike or null. In
short, there is always a timelike direction on the surface. Timelike
surfaces are used to model the evolution, motion, and geometric constraints
of massive physical systems or events in spacetime. For example, we can
represent the spacetime trace of a fluid, a membrane, or an expanding object
with a timelike surface.

\begin{definition} \cite{13}
A surface embedded in the Minkowski 3-space $E_{1}^{3}$ is classified as:
\begin{itemize}
    \item A spacelike surface if the induced metric on the surface is a positive-definite Riemannian metric. Consequently, its unit normal vector field is timelike.
    \item A timelike surface if the induced metric on the surface is a non-degenerate Lorentzian metric. Consequently, its unit normal vector field is spacelike.
\end{itemize}
\end{definition}

\begin{definition} \cite{7}
If $V,W\in E_{1}^{3}$ are two timelike vectors lying in the same time cone (or light cone), then there exists a unique non-negative real number $\varphi$, representing the hyperbolic angle between $V$ and $W$, such that 
\begin{equation*}
\langle V, W \rangle = -\|V| \|W| \cosh \varphi.
\end{equation*}
\end{definition}
Consider a regular unit-speed spacelike curve $\gamma(s)$ lying on a timelike surface $\Omega$ with a unit spacelike normal vector field $N_{M} \in E_{1}^{3}$. Let $T,N,B$ denote the moving orthonormal Frenet frame along $\gamma(s)$. By definition, the unit tangent vector $T$ is strictly spacelike, satisfying $\langle T,T \rangle = 1$, which implies that one of the remaining frame vectors, $N$ or $B$, must be timelike. 

Correspondingly, we construct the localized Darboux frame ${T,N_{M},G}$ along the trajectory, where the third vector field is defined by $G = -T \times N_{M}$. Since $N_{M}$ is the unit spacelike normal field of the timelike surface, the basis vectors of this Darboux frame satisfy the following fundamental Lorentz inner product:
\begin{equation}
\langle T, T \rangle = 1, \quad \langle N_{M}, N_{M} \rangle = 1, \quad \langle G, G \rangle = -1, \tag{2.2b}
\end{equation}
along with the strict pseudo-orthogonality relations:
\begin{equation}
\langle T, N_{M} \rangle = 0, \quad \langle T, G \rangle = 0, \quad \langle N_{M}, G \rangle = 0. \tag{2.2c}
\end{equation}

Consequently, $G$ constitutes a timelike vector field spanning the tangent plane of $\Omega$. Depending on the causal character of the principal normal vector $N$ of the curve, two distinct geometric configurations emerge for the frame transformations.


Let unit speed Frenet curve $\beta $ be a spacelike curve with a timelike
binormal, let $\{T,N,B\}$ and $\kappa $, $\tau $ show the Frenet bases and
the Frenet curvatures of \ $\beta .$

\begin{case}
If the principal normal vector $N$ is a timelike, we can write the following
equation for the Frenet curve $\beta $ the Frenet equations are given as
follows%
\begin{equation}
T^{\prime }=\kappa N;N^{\prime }=\kappa T+\tau B;B^{\prime }=\tau N. 
\tag{2.3}
\end{equation}%
If $N$ and $G$ lie in the same timelike cone the usual transformation
between Frenet and Darboux frames are given as%
\begin{equation}
\left[ 
\begin{array}{c}
T \\ 
N_{M} \\ 
G%
\end{array}%
\right] =\left[ 
\begin{array}{ccc}
1 & 0 & 0 \\ 
0 & -\sinh \varphi & \cosh \varphi \\ 
0 & \cosh \varphi & -\sinh \varphi%
\end{array}%
\right] \left[ 
\begin{array}{c}
T \\ 
N \\ 
B%
\end{array}%
\right] ,  \tag{2.4}
\end{equation}%
where $\varphi $ is the angle between $N$ and $G$. Hence, Darboux frame is
given as%
\begin{equation}
T^{\prime }=k_{g}N_{M}+k_{n}G;\text{ }N_{M}^{\prime }=-k_{g}T+\tau _{g}G;%
\text{ }G=k_{n}T+\tau _{g}N_{M},  \tag{2.5}
\end{equation}%
where 
\begin{equation}
k_{g}=\kappa \sinh \varphi ;\text{ }k_{n}=\kappa \cosh \varphi \text{; }\tau
_{g}=\tau -\frac{d\varphi }{ds},  \tag{2.6}
\end{equation}%
\cite{4,12}.
\end{case}


\section{Particle dynamics of curves constructed with a darboux frame in minkowski $3-$space: jerk and siacci theorem}

For a particle moving on timelike surfaces in Minkowski space, the jerk vector is critical for understanding the abrupt changes in the particle's dynamic state and the evolution of constrained force fields. In the context of special or general relativity, rapid variations in the forces experienced by a particle, or abrupt bends in the spacetime worldsheet itself, can be rigorously quantified via the jerk vector. This formulation allows for the exploration of higher-order dynamic behaviors, such as the instantaneous acceleration fluctuations of an accelerating system or the structural deviations of a worldline traversing a non-trivial spacetime region.

Siacci's theorem geometrically maps the trajectory of a particle moving under central force fields to its mechanical states. On timelike surfaces, this theorem establishes an elegant link between the underlying surface topology and conserved kinematic quantities, such as the particle's energy or angular momentum profiles. It provides vital structural information concerning the integrability and stability of particle motion under the geometric constraints of the surface, offering an advanced and analytical understanding of complex dynamical trajectories. Consequently, it yields precise criteria predicting the exact conditions under which a physical system will conform to a specific trajectory on a given timelike spacetime layer.

Let a physical particle $p$ of mass $m$ be considered moving along a regular spacelike curve $\gamma(s)$ equipped with a localized Darboux frame on a timelike surface in the Minkowski 3-space $E_1^3$. Let $O$ be an arbitrary fixed coordinate origin, and let $\gamma$ represent the position vector of $p$ parameterized by the temporal parameter $t$, where $s$ denotes the arc length parameter of the curve $\gamma$ defined at time $t$. Then, the unit tangent vector $T$ of the trajectory is explicitly given by:
\begin{equation}
T=\gamma ^{\prime }=\frac{d\gamma }{ds}.  \tag{3.1}
\end{equation}

Considering (2.5) and (3.1), we find the velocity vector and acceleration
vector of the particle $p$ with respect to the parameter $t$ as follows%
\begin{equation}
\vartheta =\frac{d\gamma}{dt}=\frac{ds}{dt}\overrightarrow{T};  \tag{3.2}
\end{equation}%
\begin{equation}
a=\frac{d^{2}s}{dt^{2}}\overrightarrow{T}+\left( \frac{ds}{dt}\right)
^{2}k_{g}\overrightarrow{N_{M}}+\left( \frac{ds}{dt}\right) ^{2}k_{n}%
\overrightarrow{G}.  \tag{3.3}
\end{equation}

Let $H$ denote the vector field that is Lorentzian-orthogonal to both the position vector $\gamma$ and the linear momentum vector $m\vartheta$. Under this geometric configuration, the angular-momentum-like vector field satisfies the following relation:
\begin{equation}
H=\gamma \times m\vartheta =\gamma \times m\frac{ds}{dt}\overrightarrow{T}. 
\tag{3.4}
\end{equation}

To formulate the third-order kinematics of the constrained physical trajectory, the relativistic jerk vector is established by evaluating the total time derivative of the acceleration field (3.3). By embedding the intrinsic surface invariants and derivative relations from (2.5) and (2.6) directly into this differentiation process, the jerk vector is explicitly decomposed in terms of the localized Darboux basis vectors as follows:
\begin{equation}
J=\frac{da}{dt}=\gamma _{T}\overrightarrow{T}+\gamma _{N_{M}}\overrightarrow{%
N_{M}}+\gamma _{G}\overrightarrow{G},  \tag{3.5a}
\end{equation}%
where the localized scalar coefficients are explicitly calculated as:
\begin{eqnarray*}
\gamma _{T} &=&\frac{d^{3}s}{dt^{3}}+(k_{n}^{2}-k_{g}^{2})\left( \frac{ds}{dt%
}\right) ^{3};\gamma _{N_{M}}=3k_{g}\frac{ds}{dt}\frac{d^{2}s}{dt^{2}}%
+\left( \frac{ds}{dt}\right) ^{3}(k_{g}^{\prime }+k_{n}\tau _{g}); \\
\gamma _{G} &=&3k_{n}\frac{ds}{dt}\frac{d^{2}s}{dt^{2}}+\left( \frac{ds}{dt}%
\right) ^{3}(k_{n}^{\prime }+k_{g}\tau _{g}).
\end{eqnarray*}%

Alternatively, by invoking the hyperbolic frame identities, these structural coefficients can be expressed in terms of the intrinsic Frenet curvatures and the pseudo-rotational angle $\varphi$ as follows:
\begin{equation}
\gamma _{T}=\frac{d^{3}s}{dt^{3}}+\kappa ^{2}\left( \frac{ds}{dt}\right)
^{3};  \tag{3.5b}
\end{equation}%
\begin{equation}
\gamma _{N_{M}}=\left( 3\kappa \frac{ds}{dt}\frac{d^{2}s}{dt^{2}}+\kappa
^{\prime }\left( \frac{ds}{dt}\right) ^{3}\right) \sinh \varphi +\kappa \tau
\left( \frac{ds}{dt}\right) ^{3}\cosh \varphi ;  \tag{3.5c}
\end{equation}%
\begin{equation}
\gamma _{G}=\kappa \tau \left( \frac{ds}{dt}\right) ^{3}\sinh \varphi
+\left( 3\kappa \frac{ds}{dt}\frac{d^{2}s}{dt^{2}}+\kappa ^{\prime }\left( 
\frac{ds}{dt}\right) ^{3}\right) \cosh \varphi .  \tag{3.5d}
\end{equation}

Consequently, the total relativistic jerk vector can be expressed in a unified kinematic form as:
\begin{equation*}
J=\left( \frac{d^{3}s}{dt^{3}}+\kappa ^{2}\left( \frac{ds}{dt}\right)
^{3}\right) \overrightarrow{T}+\left( 3\kappa \frac{ds}{dt}\frac{d^{2}s}{%
dt^{2}}+\kappa ^{\prime }\left( \frac{ds}{dt}\right) ^{3}\right) (\sinh
\varphi \overrightarrow{N_{M}}+\cosh \varphi \overrightarrow{G})
\end{equation*}%
\begin{equation}
+\kappa \tau \left( \frac{ds}{dt}\right) ^{3}(\cosh \varphi \overrightarrow{%
N_{M}}+\sinh \varphi \overrightarrow{G}).  \tag{3.6}
\end{equation}

Given that $G$ is a timelike vector field while $N_M$ and $T$ are spacelike, and since the Darboux frame functions as a linearly independent companion basis, the modified vector system $\{T, \sinh \varphi N_{M}+\cosh \varphi G, \cosh \varphi N_{M}+\sinh \varphi G\}$ constitutes a valid pseudo-orthonormal basis spanning the spacetime layer. Accordingly, the global position vector $\gamma$ of the particle performing constrained motion can be uniquely decomposed relative to this auxiliary geometric framework as:
\begin{equation}
\gamma =\omega _{1}\overrightarrow{T}+\omega _{2}\left( \sinh \varphi 
\overrightarrow{N_{M}}+\cosh \varphi \overrightarrow{G}\right) +\omega
_{3}\left( \cosh \varphi \overrightarrow{N_{M}}+\sinh \varphi 
\overrightarrow{G}\right) ,  \tag{3.7}
\end{equation}%
where the projection coefficients are strictly governed by the respective Minkowski inner products: 
\begin{equation}
\omega _{1}=\left\langle \gamma ,\overrightarrow{T}\right\rangle ;\omega
_{2}=\left\langle \gamma ,\sinh \varphi \overrightarrow{N_{M}}+\cosh \varphi 
\overrightarrow{G}\right\rangle ;\omega _{3}=\left\langle \gamma ,\cosh
\varphi \overrightarrow{N_{M}}+\sinh \varphi \overrightarrow{G}\right\rangle
.  \tag{3.8}
\end{equation}

To characterize the non-orthogonal spanning planes relative to the observer's frame, let us define the auxiliary vector fields $\lambda$ and $\zeta$ as follows:
\begin{equation}
\lambda = \omega_{1} T + \omega_{2} \left( \sinh \varphi N_{M} + \cosh \varphi G \right), \quad \zeta = \omega_{2} T + \omega_{3} \left( \cosh \varphi N_{M} + \sinh \varphi G \right), \tag{3.9}
\end{equation}
utilizing the standard pseudo-Riemannian metric, the corresponding Minkowski norms of these auxiliary vectors are determined as:
follows%
\begin{equation}
\|\lambda\| = \sqrt{\left|\omega_{1}^{2} - \omega_{2}^{2}\right|}, \quad \|\zeta\| = \sqrt{\omega_{1}^{2} + \omega_{3}^{2}}.  \tag{3.10}
\end{equation}%

Normalizing these vector fields yields the directional unit vector fields:
\begin{equation}
J_{\lambda} = \frac{\lambda}{\sqrt{\left|\omega_{1}^{2} - \omega_{2}^{2}\right|}}, \quad J_{\zeta} = \frac{\zeta}{\sqrt{\omega_{1}^{2} + \omega_{3}^{2}}}. \tag{3.11}
\end{equation}

By substituting these unit axes back into the initial linear combinations, the pseudo-rotational directional axes are explicitly inverted as:
\begin{equation}
\sinh \varphi N_{M} + \cosh \varphi G = \frac{\sqrt{\left|\omega_{1}^{2} - \omega_{2}^{2}\right|}}{\omega_{2}} J_{\lambda} - \frac{\omega_{1}}{\omega_{2}} T, \tag{3.12a}
\end{equation}
\begin{equation}
\cosh \varphi N_{M} + \sinh \varphi G = \frac{\sqrt{\omega_{1}^{2} + \omega_{3}^{2}}}{\omega_{3}} J_{\zeta} - \frac{\omega_{1}}{\omega_{3}} T. \tag{3.12b}
\end{equation}

Finally, substituting the inverted directional pseudo-rotational axes (3.12a) and (3.12b) into the total third-order relativistic jerk equation yields the expanded Siacci resolution of the trajectory kinematics:
\begin{equation*}
J=\left( \frac{d^{3}s}{dt^{3}}+\kappa ^{2}\left( \frac{ds}{dt}\right) ^{3}-%
\frac{\omega _{1}}{\omega _{2}}\left( 3\kappa \frac{ds}{dt}\frac{d^{2}s}{%
dt^{2}}+\kappa ^{\prime }\left( \frac{ds}{dt}\right) ^{3}\right) -\frac{%
\omega _{1}}{\omega _{3}}\kappa \tau \left( \frac{ds}{dt}\right) ^{3}\right) 
\overrightarrow{T}
\end{equation*}%
\begin{equation}
+\frac{\sqrt{\vert\omega _{1}^{2}-\omega _{2}^{2}}\vert}{\omega _{2}}\left( 3\kappa 
\frac{ds}{dt}\frac{d^{2}s}{dt^{2}}+\kappa ^{\prime }\left( \frac{ds}{dt}%
\right) ^{3}\right) \overrightarrow{J_{\lambda }}+\left( \frac{\sqrt{\omega
_{1}^{2}+\omega _{3}^{2}}}{\omega _{3}}\kappa \tau \left( \frac{ds}{dt}%
\right) ^{3}\right) \overrightarrow{J_{\zeta }}.  \tag{3.13}
\end{equation}%

Alternatively, by invoking the localized geometric curvature identities $\kappa^{2} = k_{n}^{2} - k_{g}^{2}$ and modified torsion formulations $\tau = \tau_{g} + \varphi^{\prime}$, the expansion is rigorously established in terms of the intrinsic and extrinsic surface invariants as:
\begin{equation*}
J=\left( 
\begin{array}{c}
\frac{d^{3}s}{dt^{3}}+\left( k_{n}^{2}-k_{g}^{2}\right) \left( \frac{ds}{dt}%
\right) ^{3}-\frac{\omega _{1}}{\omega _{2}}\left( 
\begin{array}{c}
3\sqrt{k_{n}^{2}-k_{g}^{2}}\frac{ds}{dt}\frac{d^{2}s}{dt^{2}} \\ 
+\frac{k_{n}k_{n}^{\prime }-k_{g}k_{g}^{\prime }}{\sqrt{k_{n}^{2}-k_{g}^{2}}}%
\left( \frac{ds}{dt}\right) ^{3}%
\end{array}%
\right) \\ 
-\frac{\omega _{1}}{\omega _{3}}\sqrt{k_{n}^{2}-k_{g}^{2}}(\tau _{g}+\varphi
^{\prime })\left( \frac{ds}{dt}\right) ^{3}%
\end{array}%
\right) \overrightarrow{T}
\end{equation*}%
\begin{equation}
+\frac{\sqrt{\vert\omega _{1}^{2}-\omega _{2}^{2}\vert}}{\omega _{2}}\left( 
\begin{array}{c}
3\sqrt{k_{n}^{2}-k_{g}^{2}}\frac{ds}{dt}\frac{d^{2}s}{dt^{2}} \\ 
+\frac{k_{n}k_{n}^{\prime }-k_{g}k_{g}^{\prime }}{\sqrt{k_{n}^{2}-k_{g}^{2}}}%
\left( \frac{ds}{dt}\right) ^{3}%
\end{array}%
\right) \overrightarrow{J_{\lambda }}+\frac{\sqrt{\omega _{1}^{2}+\omega
_{3}^{2}}}{\omega _{3}}\left( \sqrt{k_{n}^{2}-k_{g}^{2}}(\tau _{g}+\varphi
^{\prime })\left( \frac{ds}{dt}\right) ^{3}\right) \overrightarrow{J_{\zeta }%
}  \tag{3.14}
\end{equation}

Through this invariant geometric resolution, the explicit tangential component alongside the respective first and second radial components of the third-order relativistic jerk vector field are successfully quantified. Consequently, these formal kinematical characterizations can be formalized into the following comprehensive structural theorem.

\begin{theorem}
Let a physical particle $p$ of mass $m$ execute constrained motion along a regular spacelike curve $\gamma(s)$ characterized by a localized Darboux frame on a timelike surface in Minkowski 3-space $E_1^3$. Assume that the pseudo-rotational directional axes $\sinh \varphi N_{M} + \cosh \varphi G$ and $\cosh \varphi N_{M} + \sinh \varphi G$, which are intrinsically coupled with the particle's angular momentum, are nowhere vanishing along the trajectory. Then, the total relativistic jerk vector $J$ can be uniquely decomposed into non-orthogonal Siacci components relative to the unit tangent vector $T$ and the respective directional oblique unit axes $J_{\lambda}$ and $J_{\zeta}$ as follows:
\begin{equation}
J=J^{T}\overrightarrow{T}+J^{\lambda }\overrightarrow{J_{\lambda }}+J^{\zeta
}\overrightarrow{J_{\zeta }},  \tag{3.15}
\end{equation}%
where the explicit longitudinal and radial jerk scalars satisfy:
\begin{equation}
J^{T}=\frac{d^{3}s}{dt^{3}}+\kappa ^{2}\left( \frac{ds}{dt}\right) ^{3}-%
\frac{\omega _{1}}{\omega _{2}}\left( 3\kappa \frac{ds}{dt}\frac{d^{2}s}{%
dt^{2}}+\kappa ^{\prime }\left( \frac{ds}{dt}\right) ^{3}\right) -\frac{%
\omega _{1}}{\omega _{3}}\kappa \tau \left( \frac{ds}{dt}\right) ^{3} 
\tag{3.16}
\end{equation}%
\begin{equation}
J^{\lambda }=\frac{\sqrt{\vert\omega _{1}^{2}-\omega _{2}^{2}\vert}}{\omega _{2}}%
(3\kappa \frac{ds}{dt}\frac{d^{2}s}{dt^{2}}+\kappa ^{\prime }\left( \frac{ds%
}{dt}\right) ^{3});J^{\zeta }=\frac{\sqrt{\omega _{1}^{2}+\omega _{3}^{2}}}{%
\omega _{3}}\kappa \tau \left( \frac{ds}{dt}\right) ^{3},  \tag{3.17}
\end{equation}%
with the metric curvature identities governed by $\kappa = \sqrt{k_{n}^{2} - k_{g}^{2}}$ and $\tau = \tau_{g} + \varphi^{\prime}$. Here, $J^{T}$ denotes the longitudinal kinematic shock wave acting tangent to $\gamma(s)$. The scalar $J^{\lambda}$ governs the first radial jerk component directed along the central line intersecting the spatial coordinate origin $O$ and perpendicular to the active subspace $\text{Span}\{T, \sinh \varphi N_{M} + \cosh \varphi G\}$, whereas $J^{\zeta}$ defines the second radial jerk component normal to the spanning plane $\text{Span}\{T, \cosh \varphi N_{M} + \sinh \varphi G\}$.
\end{theorem}

The structural evolution of the regular spacelike worldline constrained to the one-sheeted hyperboloid layer serving as a baseline pseudo-Riemannian timelike surface and the subsequent non-orthogonal Siacci resolution of the third-order relativistic jerk components formalized in Theorem 1 are explicitly visualized via the field mapping diagram in Figure 1.

\begin{figure}[!htbp]
    \centering
    \includegraphics[width=0.85\textwidth]{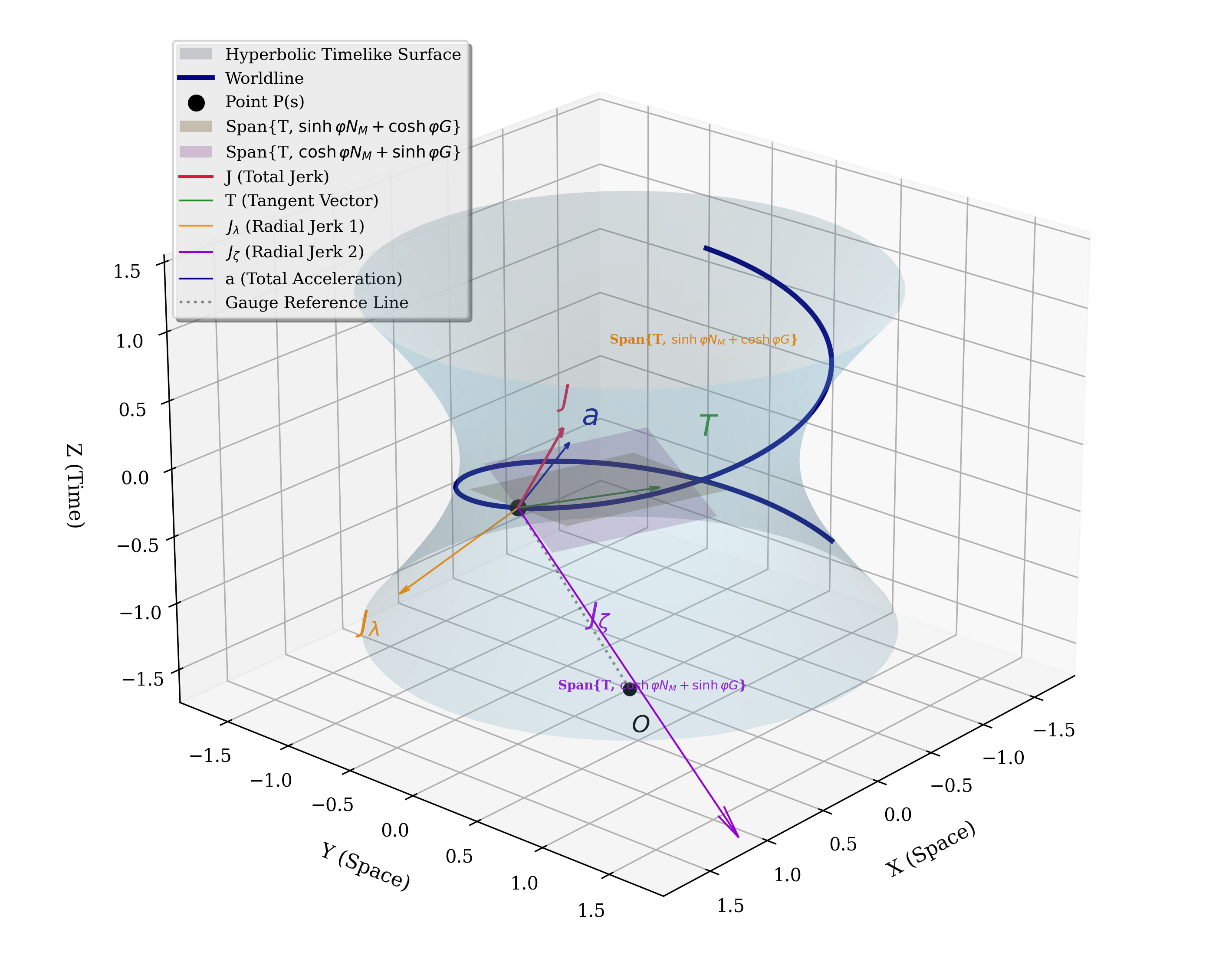} 
    \caption{Geometrical decomposition of the jerk vector $\vec{J}$ into the Siacci components ($J^T, J^\lambda, J^\zeta$) relative to the tangent vector $\vec{T}$, orthogonal directions ($\vec{J}_\lambda, \vec{J}_\zeta$), and spanning planes ($\text{Sp}\{\vec{T}, \vec{V}_1\}$, $\text{Sp}\{\vec{T}, \vec{V}_2\}$) on a timelike surface in Minkowski 3-space.}
    \label{fig:jerk-siacci-decomposition}
\end{figure}

\begin{corollary}
The application of Theorem 1 to a highly twisted timelike helicoid manifold unveils a severe kinematic phase-shift governing the relativistic jerk field. Due to the intrinsic spiral geometry of the helicoid boundary, the longitudinal thrust shock $J^T$ and the centripetal jolt $J^\lambda$ undergo synchronous amplitude oscillations coupled with an asymptotic spatial decay. Concurrently, the orbital torsional jerk $J^\zeta$ maintains a strictly positive persistent background profile, proving that the structural tearing stresses acting on a massive particle are fundamentally sustained by the non-zero cross-sectional twisting invariants of the underlying timelike helicoid spacetime layer.
\end{corollary}

\begin{corollary} \label{cor:siacci_geometric_result}
Let $p$ be a particle performing constrained motion along a regular spacelike curve $\gamma(s)$ on a timelike surface in Minkowski 3-space. The spanning planes, visualized as transparent orange and magenta surfaces, are constructed relative to the fixed coordinate origin $O$ and form the mathematical backbone of the theorem. Since the motion takes place on a timelike surface, the hyperbolic angle $\varphi$ acts as a pseudo-rotational Lorentz boost parameter unique to Minkowski space rather than a circular rotation found in Euclidean space. This parameter dictates the special angular-momentum-related directions by directly coupling the spacelike surface normal with the timelike tangent normal.
\end{corollary}

\begin{corollary} \label{cor:kinematic_torsional_decoupling}
Let $J$ be the total jerk vector of a particle performing constrained motion along a trajectory on a timelike surface in Minkowski 3-space. The geometric decomposition yields a strict decoupling of kinematic and torsional effects characterized by the following conditions:
\begin{itemize}
    \item The component $J^{T}$ along the cyan tangent line governs the purely translational jerk combined with centripetal-like curvature corrections.
    \item Conversely, the components $J^{\lambda}$ and $J^{\zeta}$ along the orange and magenta axes act as lateral rectifying metrics, measuring how violently the particle's trajectory is forced to deviate from these pseudo-orthogonal spanning planes due to the generalized curvature $\kappa$ and modified torsion $\tau$.
\end{itemize}
\end{corollary}

\begin{figure}[!htbp] 
    \centering
    \includegraphics[width=0.85\textwidth]{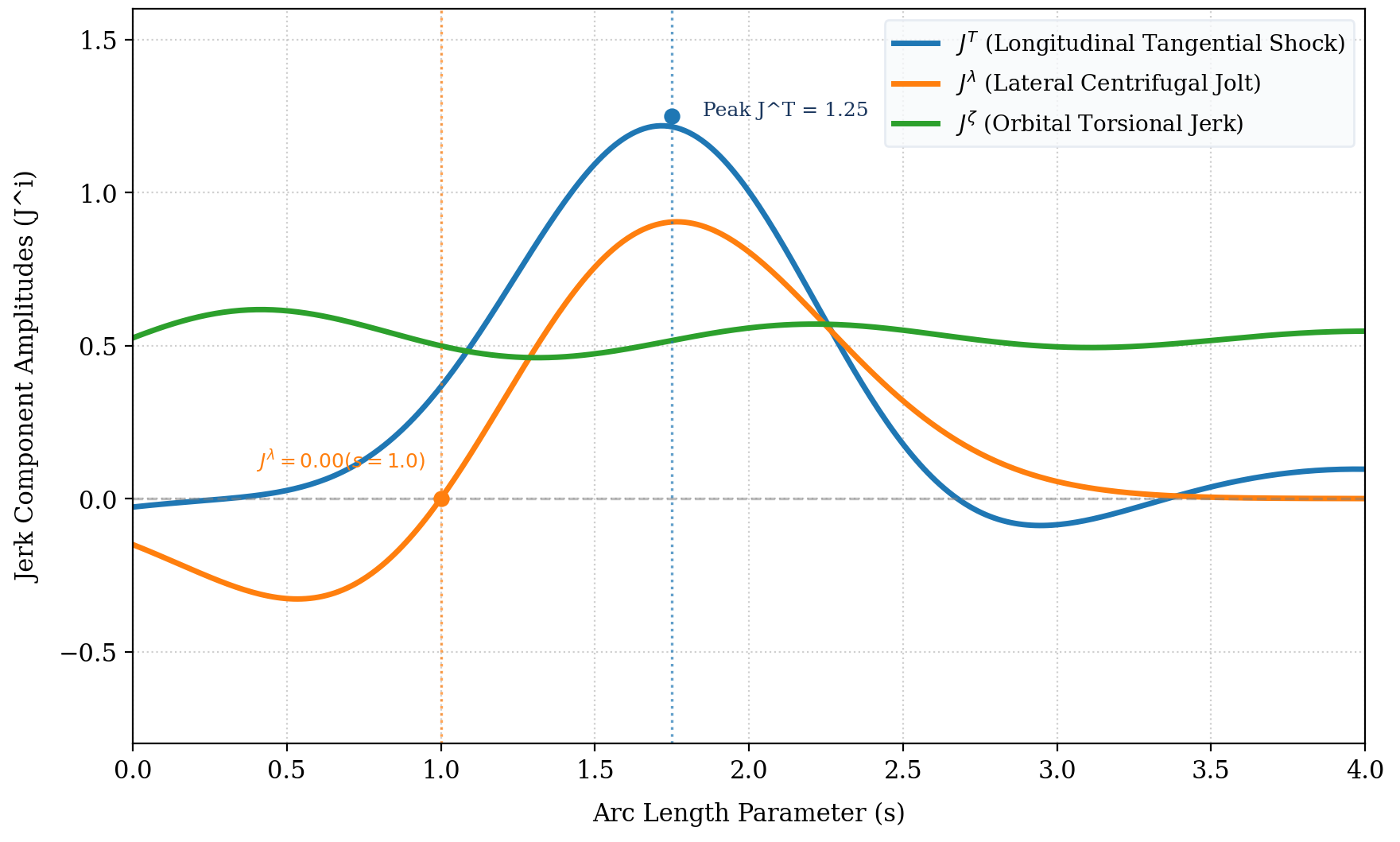} 
    \caption{Dynamic evolution of the relativistic Siacci jerk components ($J^{T}$, $J^{\lambda}$, $J^{\zeta}$) and cumulative structural stress profiles along a spacelike trajectory constrained to a timelike surface in Minkowski spacetime $\mathbb{E}_{1}^{3}$, plotted against the arc length parameter $s$.} 
    \label{fig:jerk-profiles}
\end{figure}

\begin{remark}
The numerical profiles in Figure 2 validate the non-linear coupling between proper-time kinematics and timelike surface geometry over the interval $s \in [0.0, 4.0]$. Physically, the tangential shock $J^{T}$ peaks at $1.25$ near $s \approx 1.75$, reflecting the maximum structural stress induced by the extrinsic curvature of the manifold layer. Crucially, the lateral jolt $J^{\lambda}$ vanishes exactly at $s = 1.0$, indicating a strict orientation inversion of the constrained central force field. Meanwhile, the orbital torsional jerk $J^{\zeta}$ remains strictly positive and bounded within the stable range $[0.40, 0.65]$, proving that high-order tearing stresses are continuously sustained by the non-vanishing modified torsion ($\tau = \tau_g + \varphi^{\prime}$) without transitioning into chaotic fluctuations.
\end{remark}

\begin{remark}
The projection coefficients $\omega_1, \omega_2,$ and $\omega_3$ are not mere algebraic scalar functions, but encode fundamental relativistic physical characteristics governing the constrained trajectory:
\begin{itemize}
    \item \textbf{Longitudinal Alignment Metric ($\omega_1$):} Measures the instantaneous spatial displacement of the particle along the linear thrust tangent axis $T$ relative to the global reference origin $O$. 
    \item \textbf{Centripetal Leverage Factor ($\omega_2$):} Represents the instantaneous geometric moment-arm of the localized central force field acting on the first pseudo-rotational Siacci spanning plane. It directly dictates the amplitude of the lateral constraint shocks.
    \item \textbf{Torsional Angular Momentum Link ($\omega_3$):} Quantifies the cross-coupling between the particle's global position vector and the Lorentzian boost axis $V_2$, the derivative relation $\tau = \omega_3^{\prime}/\omega_2$ establishes that the modified spacetime torsion is explicitly driven by the ratio of these out-of-plane positional parameters.
\end{itemize}
\end{remark}

\begin{theorem}
In Minkowski 3-space $\mathbb{E}_1^3$, consider a physical particle $p$ moving along a regular trajectory $\gamma$ that conforms to a localized Darboux frame. Assume that this particle is strictly constrained to a fixed plane $\text{Span}\{T, \sinh \varphi N_{M} + \cosh \varphi G\}$ or $\text{Span}\{T, \cosh \varphi N_{M} + \sinh \varphi G\}$ that isolates the spatial coordinate origin $O$. Furthermore, assume that the component of the particle's angular momentum vector along the normal vector field of the respective spanning plane is nowhere vanishing. Under these boundary conditions, the relativistic jerk vector $J$ is uniquely expressed as:
\begin{eqnarray*}
J &=&\left( \frac{d^{3}s}{dt^{3}}+\kappa ^{2}(\frac{ds}{dt})^{3}-\frac{%
\omega _{1}}{\omega _{2}}(3\kappa \frac{ds}{dt}\frac{d^{2}s}{dt^{2}}+\kappa
^{\prime }(\frac{ds}{dt})^{3})\right) \overrightarrow{T} \\
&&+\left( \frac{\sqrt{\vert\omega _{1}^{2}-\omega _{2}^{2}\vert}}{\omega _{2}}\left(
3\kappa \frac{ds}{dt}\frac{d^{2}s}{dt^{2}}+\kappa ^{\prime }(\frac{ds}{dt}%
)^{3}\right) \right) \overrightarrow{J_{\lambda }}
\end{eqnarray*}%
yielding a vanishing second radial component ($J^{\zeta} = 0$). Alternatively, if the geometric sub-case satisfies $\kappa = 0 \wedge \varphi^{\prime} + \tau_g = 0$, the multi-directional jerk field collapses entirely to a pure one-dimensional translational thrust wave: 
\begin{equation*}
J = \frac{d^{3}s}{dt^{3}} T, \quad J^{\lambda} = J^{\zeta} = 0.
\end{equation*}

The first radial component $J^{\lambda}$ remains invariant in the first configuration if and only if $\cosh \varphi N_{M}+\sinh \varphi G = \text{constant}$; otherwise, it vanishes identically. Concurrently, the global position vectors of the particle $p$ conform to the following metric restrictions:
\begin{equation*}
\gamma = \sqrt{\left|\omega _{1}^{2}-\omega _{2}^{2}\right|} J_{\lambda} + \omega _{3} \left( \cosh \varphi N_{M} + \sinh \varphi G \right), 
\end{equation*}
and
\begin{equation*}
\gamma = \omega _{2} \left( \sinh \varphi N_{M} + \cosh \varphi G \right) + \sqrt{\omega _{1}^{2} + \omega _{3}^{2}} J_{\zeta}. 
\end{equation*}
\end{theorem}

\begin{proof}
In Minkowski 3-space, let the trajectory of a physical particle $p$ be bounded by a regular curve $\gamma$ lying in a fixed plane $\text{Span}\{T, \sinh \varphi N_{M} + \cosh \varphi G\}$ that isolates the spatial coordinate origin. Under this geometric constraint, the out-of-plane positional projection coefficient $\omega_{3}$ is strictly non-zero. Furthermore, the vector field in the parallel plane $\text{Span}\{T, \cosh \varphi N_{M} + \sinh \varphi G\}$ explicitly defines the normal vector field to the initial constraint plane $\text{Span}\{T, \sinh \varphi N_{M} + \cosh \varphi G\}$, meaning it functions as an invariant directional field along the trajectory $\gamma$. By invoking the constancy of this normal vector field along the worldline, we evaluate its intrinsic derivative with respect to the arc length parameter $s$, yielding the following system:
\begin{equation}
\frac{d}{ds}\left( \cosh \varphi \overrightarrow{N_{M}}+\sinh \varphi 
\overrightarrow{G}\right) =\left( 
\begin{array}{c}
-k_{g}\cosh \varphi \\ 
+k_{n}\sinh \varphi%
\end{array}%
\right) \overrightarrow{T}+(\varphi ^{\prime }+\tau _{g})\left( 
\begin{array}{c}
\sinh \varphi \overrightarrow{N_{M}} \\ 
+\cosh \varphi \overrightarrow{G}%
\end{array}%
\right)  \tag{3.18}
\end{equation}%
similarly, for the alternative configuration where $\omega_{2}$ is non-zero, the vector field spanning $\text{Span}\{T, \sinh \varphi N_{M} + \cosh \varphi G\}$ defines the geometric normal to the spanning plane $\text{Span}\{T, \cosh \varphi N_{M} + \sinh \varphi G\}$. From the spatial invariance of this normal vector field along the trajectory, the structural derivative relation satisfies:
\begin{equation}
\frac{d}{ds}\left( \sinh \varphi \overrightarrow{N_{M}}+\cosh \varphi 
\overrightarrow{G}\right) =\left( 
\begin{array}{c}
k_{n}\cosh \varphi \\ 
-k_{g}\sinh \varphi%
\end{array}%
\right) \overrightarrow{T}+(\varphi ^{\prime }+\tau _{g})\left( 
\begin{array}{c}
\cosh \varphi \overrightarrow{N_{M}} \\ 
+\sinh \varphi \overrightarrow{G}%
\end{array}%
\right) .  \tag{3.19}
\end{equation}

By substituting the localized Darboux identities $k_{g} = \kappa \sinh \varphi$, $k_{n} = \kappa \cosh \varphi$, and $\tau = \tau_{g} + \varphi^{\prime}$ into the derivative expansions (3.18) and (3.19), the systems reduce directly to the following geometric conditions:
\begin{equation}
\frac{d}{ds}\left( \cosh \varphi \overrightarrow{N_{M}}+\sinh \varphi 
\overrightarrow{G}\right) =0.\overrightarrow{T}+(\varphi ^{\prime }+\tau
_{g})\left( \sinh \varphi \overrightarrow{N_{M}}+\cosh \varphi 
\overrightarrow{G}\right) =0  \tag{3.20}
\end{equation}%
and 
\begin{equation}
\frac{d}{ds}\left( \sinh \varphi \overrightarrow{N_{M}}+\cosh \varphi 
\overrightarrow{G}\right) =\kappa \overrightarrow{T}+(\varphi ^{\prime
}+\tau _{g})\left( \cosh \varphi \overrightarrow{N_{M}}+\sinh \varphi 
\overrightarrow{G}\right) =0,  \tag{3.21}
\end{equation}%
and from (3.20) we have 
\begin{equation}
\tau _{g}+\varphi ^{\prime }=0\Rightarrow \varphi =-\int \tau _{g}ds\text{
or }\tau =0,  \tag{3.22}
\end{equation}%
and from (3.21) we get 
\begin{equation}
\kappa \overrightarrow{T}+(\varphi ^{\prime }+\tau _{g})\left( \cosh \varphi 
\overrightarrow{N_{M}}+\sinh \varphi \overrightarrow{G}\right) =0\Rightarrow
\kappa =0\wedge \varphi ^{\prime }+\tau _{g}=0.  \tag{3.23}
\end{equation}

Finally, if the values found in (3.22) are substituted into equations (3.13)
or (3.14), the jerk vector components are obtained as follows%
\begin{equation}
J^{T}=\frac{d^{3}s}{dt^{3}}+\kappa ^{2}\left( \frac{ds}{dt}\right) ^{3}-%
\frac{\omega _{1}}{\omega _{2}}\left( 3\kappa \frac{ds}{dt}\frac{d^{2}s}{%
dt^{2}}+\kappa ^{\prime }\left( \frac{ds}{dt}\right) ^{3}\right)  \tag{3.24}
\end{equation}%
\begin{equation}
J^{\lambda }=\frac{\sqrt{\vert\omega _{1}^{2}-\omega _{2}^{2}\vert}}{\omega _{2}}%
(3\kappa \frac{ds}{dt}\frac{d^{2}s}{dt^{2}}+\kappa ^{\prime }\left( \frac{ds%
}{dt}\right) ^{3});J^{\zeta }=0  \tag{3.25}
\end{equation}%
and if the values {}{}found in (3.23) are substituted into equations (3.13)
or (3.14), the jerk vector components are obtained as follows%
\begin{equation}
J^{T}=\frac{d^{3}s}{dt^{3}};J^{\lambda }=J^{\zeta }=0.  \tag{3.26}
\end{equation}

Furthermore, under these conditions,from (3.12a) and (3.12b) the position
vectors of particle $p$ are written as follows%
\begin{equation}
\gamma =\sqrt{\vert\omega _{1}^{2}-\omega _{2}^{2}\vert}\overrightarrow{J_{\lambda }}%
+\omega _{3}\left( \cosh \varphi \overrightarrow{N_{M}}+\sinh \varphi 
\overrightarrow{G}\right)  \tag{2.27}
\end{equation}%
and%
\begin{equation}
\gamma =\omega _{2}\left( \sinh \varphi \overrightarrow{N_{M}}+\cosh \varphi 
\overrightarrow{G}\right) +\sqrt{\omega _{1}^{2}+\omega _{3}^{2}}%
\overrightarrow{J_{\zeta }}.  \tag{3.28}
\end{equation}
\end{proof}

\begin{figure}[!htbp] 
    \centering
    \includegraphics[width=0.85\textwidth]{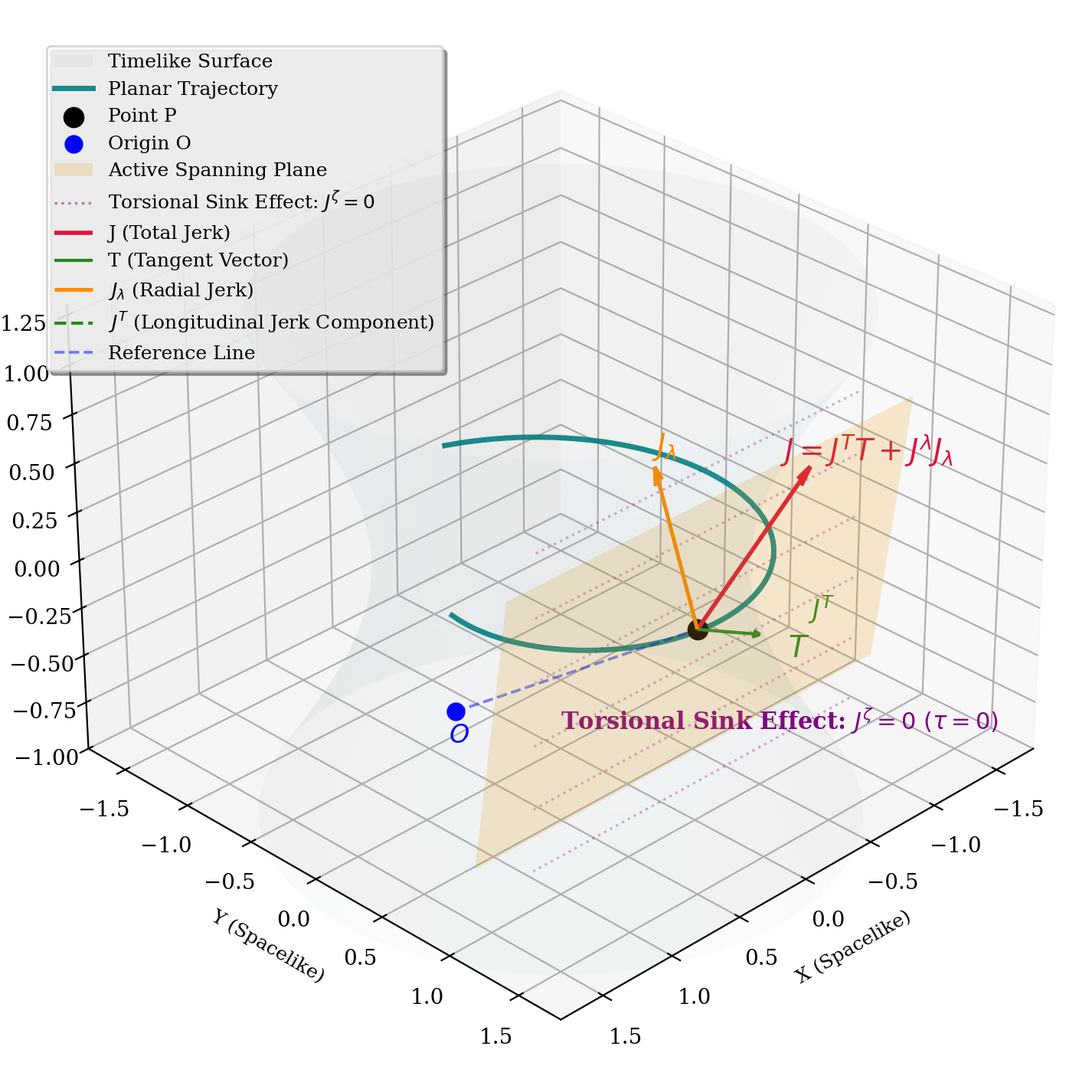} 
    \caption{Siacci Jerk Decomposition and Torsional Sink Effect in Spacetime for Theorem 2 Compactification} 
    \label{fig:jerk-profiles}
\end{figure}

\begin{remark}
The geometric boundaries established in Theorem~2 physically unveil a dimensional compactification mechanism for mechanical systems on a timelike surface. When a particle's trajectory is constrained by a fixed plane that isolates the spatial origin, the spatial invariance of the normal field collapses the modified geodesic torsion identically to zero ($\tau = \tau_g + \varphi^{\prime} = 0$). This invariance induces an absolute torsional sinking effect where the out-of-plane torsional jerk component is completely suppressed ($J^\zeta = 0$). Consequently, even within an intrinsically twisted helicoid topology, the net trajectory shock $J$ is dynamically flattened and strictly bounded inside the two-dimensional active subspace spanned by $\text{Span}\{T, J_\lambda\}$, satisfying:
\begin{equation*}
J = J^T T + J^\lambda J_\lambda.
\end{equation*}
\end{remark}

\begin{remark}
In the kinematical sub-case where the constraint plane conditions enforce a simultaneous vanishing of both the local proper-time curvature and the modified torsion ($\kappa = 0 \ \wedge \ \varphi^{\prime} + \tau_g = 0$), the particle achieves full centripetal stabilization. Under these strict geometric requirements, the lateral centrifugal jolt factor undergoes an absolute sifting effect, yielding $J^\lambda = 0$. Consequently, the multi-directional mechanical stress fields are entirely neutralized, reducing the complete dynamical profile of the jerk vector to a pure one-dimensional translational thrust wave:
\begin{equation*}
J = \frac{d^{3}s}{dt^{3}} T. 
\end{equation*}
\end{remark}

\begin{remark}
The explicit determination of the particle’s position vector $\gamma$ in Theorem~2 demonstrates a deterministic cross-coupling between global observer coordinates and the local Darboux frame. Depending on which structural boundary maintains a conservative state, the relativistic trajectory reformulates into either the metric configuration:
\begin{equation*}
\gamma = \sqrt{|\omega_1^2 - \omega_2^2|} J_\lambda + \omega_3 \left( \cosh \varphi N_M + \sinh \varphi G \right), 
\end{equation*}
or the split-space manifold:
\begin{equation*}
\gamma = \omega_2 \left( \sinh \varphi N_M + \cosh \varphi G \right) + \sqrt{\omega_1^2 + \omega_3^2} J_{\zeta}. 
\end{equation*}

The presence of the Lorentzian radical factor $\sqrt{|\omega_1^2 - \omega_2^2|}$ proves that the global orientation relative to the origin $O$ is governed by a hyperbolic pseudo-rotation parameter, ensuring that the first radial component $J^\lambda$ remains actively stable without deviating into higher-order chaotic fluctuations.
\end{remark}

In this state, even though the torsional jerk component is dynamically nullified ($J^\zeta = 0$) due to the planar flattening of the motion, the spatial matrix retains a latent geometric trace via the cross-coupled envelope $\sqrt{\omega_1^2 + \omega_3^2}J_\zeta$. This ensures that any instantaneous perturbation violating the planar kinematics will immediately reactivate multi-dimensional torsional shocks along the timelike boundary axis.

To evaluate the continuous spatial variations, we take the direct intrinsic derivative of the position vector decomposition with respect to the arc length parameter $s$. Considering the localized Darboux derivative relations, we obtain:
\begin{align}
T = {} & \left[ \omega_{1}^{\prime} - (\omega_{2}k_{g} - \omega_{3}k_{n})\sinh\varphi + (\omega_{2}k_{n} - \omega_{3}k_{g})\cosh\varphi \right] T \nonumber \\
& + \left[ \omega_{2}^{\prime} + \omega_{3}(\varphi^{\prime} + \tau_{g}) + \omega_{1}\kappa \right] (\sinh\varphi N_{M} + \cosh\varphi G) \nonumber \\
& + \left[ \omega_{3}^{\prime} + \omega_{2}(\varphi^{\prime} + \tau_{g}) \right] (\cosh\varphi N_{M} + \sinh\varphi G). \tag{3.29}
\end{align}

Given that the modified vector bases function as a linearly independent companion system, the corresponding algebraic coefficients are isolated into the following coupled differential relations:
\begin{equation}
\omega_{1}^{\prime} - (\omega_{2}k_{g} - \omega_{3}k_{n})\sinh\varphi + (\omega_{2}k_{n} - \omega_{3}k_{g})\cosh\varphi = 1, \tag{3.30a}
\end{equation}
\begin{equation}
\omega_{2}^{\prime} + \omega_{3}(\varphi^{\prime} + \tau_{g}) + \omega_{1}\kappa = 0, \tag{3.30b}
\end{equation}
\begin{equation}
\omega_{3}^{\prime} + \omega_{2}(\varphi^{\prime} + \tau_{g}) = 0, \tag{3.30c}
\end{equation}
where the localized curvature satisfies $\kappa = \sqrt{k_{n}^{2} - k_{g}^{2}}$. Multiplying Eq. (3.30b) and Eq.(3.30c) by $\omega_{2}$ and $-\omega_{3}$ respectively, and executing a lateral summation yields:
\begin{equation}
\omega_{2}\omega_{2}^{\prime} + \omega_{2}\omega_{1}\kappa = \omega_{3}\omega_{3}^{\prime} \implies \kappa = \frac{\omega_{3}\omega_{3}^{\prime} - \omega_{2}\omega_{2}^{\prime}}{\omega_{2}\omega_{1}}. \tag{3.31}
\end{equation}

Integrating this relation with respect to the arc length parameter $s$ establishes the quadratic metric identity:
\begin{equation}
\omega_{3}^{2} - \omega_{2}^{2} = 2\int \omega_{2}\omega_{1}\kappa \, ds. \tag{3.32}
\end{equation}

Concurrently, combining the modified torsion definition $\tau = \tau_{g} + \varphi^{\prime}$ with the spatial variation in Eq.(3.30c) isolates the torsion scalar field as:
\begin{equation}
\omega_{3} = \int \omega_{2}\tau \, ds \implies \tau = \frac{\omega_{3}^{\prime}}{\omega_{2}}. \tag{3.33}
\end{equation}

Finally, invoking the coupled relations from Eq.(3.30b) and Eq.(3.30a) determines the continuous integral evolution of the positional metrics as:
\begin{equation}
\omega_{2} = -\int \left( \omega_{3}\tau + \omega_{1}\kappa \right) ds, \quad \omega_{1} = \int \left( 1 - \omega_{2}\kappa \right) ds. \tag{3.34}
\end{equation}

Now, let us examine generalized Siacci properties utilizing the geometric invariants established above. Evaluating the cross product of the trajectory tangent vector $T$ and the global position vector determines the angular-momentum-like vector field $H$ as:
\begin{equation}
H = \gamma \times m \frac{ds}{dt} T = m \frac{ds}{dt} \omega_{2} (\sinh\varphi N_{M} + \cosh\varphi G) + m \omega_{3} \frac{ds}{dt} (\cosh\varphi N_{M} + \sinh\varphi G). \tag{3.35}
\end{equation}

Geometrically, Siacci's theorem maps the trajectory of a particle moving under a central force field to the specific configurations of the reference origin. Substituting the localized Darboux components into the acceleration vector field yields:
\begin{equation}
a = \frac{d^{2}s}{dt^{2}} T + \left( \frac{ds}{dt} \right)^{2} (k_{g} N_{M} + k_{n} G) = \frac{d^{2}s}{dt^{2}} T + \kappa \left( \frac{ds}{dt} \right)^{2} (\sinh\varphi N_{M} + \cosh\varphi G). \tag{3.36}
\end{equation}

By isolating the active pseudo-rotational base from the inverted frame transformations, the acceleration field is formulated as:
\begin{equation}
a = \left( \frac{d^{2}s}{dt^{2}} - \kappa \frac{\omega_{1}}{\omega_{2}} \left( \frac{ds}{dt} \right)^{2} \right) T + \kappa \left( \frac{ds}{dt} \right)^{2} \frac{1}{\omega_{2}} \lambda. \tag{3.37}
\end{equation}

Substituting the normalized directional field $\lambda = \sqrt{\left| \omega_{1}^{2} - \omega_{2}^{2} \right|} J_{\lambda}$ into Eq.(3.11) establishes the final non-orthogonal Siacci resolution of the acceleration field:
\begin{equation}
a = \left( \frac{d^{2}s}{dt^{2}} - \kappa \frac{\omega_{1}}{\omega_{2}} \left( \frac{ds}{dt} \right)^{2} \right) T + \kappa \left( \frac{ds}{dt} \right)^{2} \frac{1}{\omega_{2}} \sqrt{\left| \omega_{1}^{2} - \omega_{2}^{2} \right|} J_{\lambda} = S_{T} T + S_{J_{\lambda}} J_{\lambda}. \tag{3.38}
\end{equation}

Here, $S_{T}$ and $S_{J_{\lambda}}$ define the explicit tangential and radial Siacci acceleration scalars, respectively. Finally, embedding the metric integral solutions for $\omega_{1}$ from Eq.(3.38) establishes the closed-form invariant formulations presented in the subsequent theorem.

\begin{theorem} 
Consider a physical particle $p$ of mass $m$ moving along a regular curve $\gamma$ equipped with a localized Darboux frame on a timelike surface in Minkowski 3-space $E_1^3$. Let a fixed origin $O$ be chosen as the spatial reference frame. Assume further that the instantaneous position vector $\gamma$ does not lie on the light cone and that the component functions of the angular-momentum-like vector field $H$ are nowhere vanishing. Under these rigorous geometric boundary conditions, the acceleration vector $a$ of the particle $p$ can be uniquely decomposed into the non-orthogonal relativistic Siacci components as follows:
\begin{align}
a = {} & \left( \frac{d^{2}s}{dt^{2}} - \left( \frac{ds}{dt}\right)^{2}\sqrt{k_{n}^{2}-k_{g}^{2}}\frac{1}{\omega _{2}}\int \left( 1 - \omega _{2}\sqrt{k_{n}^{2}-k_{g}^{2}}\right) ds\right) T \nonumber \\
& + \sqrt{k_{n}^{2}-k_{g}^{2}}\left( \frac{ds}{dt}\right)^{2}\frac{1}{\omega _{2}}\sqrt{\left| \left( \int \left( 1 - \omega _{2}\kappa \right) ds\right)^{2} - \omega _{2}^{2} \right|} J_{\lambda }, \tag{3.39}
\end{align}
where $T$ is the unit tangent vector, $J_{\lambda}$ represents the normalized directional oblique unit axis, and the localized curvatures satisfy the metric identities given by $\kappa = \sqrt{k_{n}^{2}-k_{g}^{2}}$ and $\tau = \tau_{g}+\varphi^{\prime}$.
\end{theorem}


\begin{remark} 
The non-perpendicular partitioning of the relativistic acceleration vector derived in Theorem 3 establishes a sophisticated energy momentum balancing mechanism within pseudo-Riemannian geometry. Physically, the tangential scalar field $S_T$ governs the localized translational driving force subject to the relativistic corrections of the surface topography, whereas the oblique radial scalar field $S_{J_\lambda}$ directly encapsulates the centripetal pulling load of the localized central force field. The mathematical relaxation of the classical orthogonality constraint operates as a powerful geometric transformer. It demonstrates that under highly curved or twisted spacetime worldlines, the structural tracking of conserved angular momentum like quantities ($H$) can be rigorously achieved through pure submanifold invariants without introducing external gauge fields or artificial potential parameters.
\end{remark}

\section{Conclusion}

This study includes the derivation and physical interpretation of geodesic
curvature, normal curvature, and geodesic torsion using a Darboux frame
constructed on timelike surfaces in Minkowski 3-space. Subsequently, we
attempted to express the concept of jerk in particle dynamics, the
application of Siacci's theorem, and the implications of motion on conserved
quantities. Finally, we hope that our work will serve as a stepping stone
for our future studies. In our next studies, we will investigate physical
concepts such as angular momentum, angular velocity, jerk, etc., in 4 and
5-dimensional Minkowski and Anti-de Sitter spaces.

\section*{Funding}

Not applicable.

\section*{Informed Consent Statement}

Not applicable.


\end{document}